\documentclass[11pt]{article}
\usepackage{latexsym}
\usepackage{amsfonts}
\usepackage{enumerate,graphicx}
\usepackage{verbatim}

\usepackage{amsmath, amssymb, amsthm}

\theoremstyle{definition}

\numberwithin{equation}{section}

\newcommand\vanish[1]{}	% \vanish{TEXT} hides text in the compiled file.

\newcommand\ourcomment[1]{ \textbf{[#1]} }
\newcommand\oc\ourcomment

\newtheorem{theorem}{Theorem}[section]

\newtheorem{lemma}{Lemma}[section]
\newtheorem{corollary}{Corollary}[section]

\title{Perfect divisibility and $2$-divisibility}
\author{Maria Chudnovsky
\thanks{Partially supported by NSF grant DMS-1550991 and by US Army Research 
Office grant W911NF-16-1-0404.}\\
Princeton University, Princeton, NJ 08544, USA
\and
Vaidy Sivaraman\\
Binghamton University, Binghamton, NY 13902, USA
}

\begin{document}
\maketitle
\begin{abstract}
A graph $G$ is said to be $2$-divisible if for all (nonempty) induced subgraphs $H$ of $G$, $V(H)$ can be partitioned into two sets $A,B$ such that $\omega(A) < \omega(H)$  and $\omega(B) < \omega(H)$. A graph $G$ is said to be perfectly divisible if for all induced subgraphs $H$ of $G$, $V(H)$ can be partitioned into two sets $A,B$ such that $H[A]$ is perfect and $\omega(B) < \omega(H)$. We prove that if a graph is $(P_5,C_5)$-free, then it is $2$-divisible. We also prove that if a graph is bull-free and either odd-hole-free or $P_5$-free, then it is perfectly divisible. \\
\end{abstract}

\maketitle

\section{Introduction}
All graphs considered in this article are finite and simple.  Let $G$ be a 
graph. The complement $G^c$ of $G$ is the graph with vertex set $V(G)$ and 
such that two vertices are adjacent in $G^c$ if and only if they are 
non-adjacent in $G$. 
For two graphs $H$ and $G$, $H$ is an {\em induced subgraph} of $G$ if 
$V(H) \subseteq V(G)$, and a pair of vertices $u,v \in V(H)$ is adjacent if and only if it is adjacent in $G$. We say that $G$ {\em contains} $H$ if $G$ has an induced subgraph isomorphic to $H$. If $G$ does not contain $H$, we say that $G$
is {\em $H$-free}. For a set $X \subseteq V(G)$ we denote by 
$G[X]$ the induced subgraph of $G$ with vertex set $X$.
For an integer $k>0$, we denote by $P_k$ the path on $k$ vertices, and
by $C_k$ the cycle on $k$ vertices. 
A {\em path in a graph} is a sequence $p_1-\ldots-p_k$ (with $k \geq 1$) of distinct vertices such that $p_i$ is adjacent to $p_j$ if and only if $|i-j|=1$. 
Sometimes we say that $p_1- \ldots -p_k$ {\em is a $P_k$}.
%We say that the {\em length} of this path is $k-1$.  
A {\em hole} in a graph is an induced subgraph that is isomorphic to  the cycle $C_k$ with $k\geq 4$, and $k$ is the {\em length} of the hole. A hole is {\em odd} if $k$ is odd, and {\em even} otherwise. The vertices of a hole can be 
numbered $c_1, \ldots, c_k$ so that $c_i$ is adjacent to $c_j$ if and only if
$|i-j| \in \{1,k-1\}$; sometimes we write $C=c_1-\ldots-c_k-c_1$. 
An {\em antihole} in a  graph is an induced subgraph that is isomorphic to  
$C_k^c$ with $k\geq 4$, and again $k$ is the {\em length} of the antihole. 
Similarly, an  antihole is {\em odd} if $k$ is odd, and {\em even} otherwise.
The {\em bull} is the graph consisting of a triangle with two disjoint pendant edges. A graph is {\em bull-free} if no induced subgraph of it is isomorphic to the bull.  The chromatic number of a graph $G$ is denoted by $\chi(G)$ and the clique number by $\omega(G)$. A graph $G$ is called {\em perfect} if for every 
induced subgraph $H$ of $G$, $\chi(H) = \omega (H)$.  
 For a set $X$ of vertices, we will usually write $\chi(X)$ instead of $\chi(G[X])$, and $\omega(X)$ instead of $\omega(G[X])$.
 If $X$ is a set of vertices and $x$ is a vertex, we will write $X + x$ for $X \cup \{x\}$.  \\

%A class of graphs is said to be hereditary if every induced subgraph of every graph in the class is also in the class.  
%We call $p_1$ and $p_k$ the {\em ends} of $P$, and write $P^*=V(P) \setminus \{p_1,p_k\}$.

A graph $G$ is said to be {\em $2$-divisible} if for all (nonempty) induced subgraphs $H$ of $G$, $V(H)$ can be partitioned into two sets $A,B$ such that $\omega(A) < \omega(H)$  and $\omega(B) < \omega(H)$.  Ho\`{a}ng and McDiarmid \cite{HM} defined the notion of $2$-divisibility.  They actually conjecture that a graph is $2$-divisible if and only if it is odd-hole-free. A graph is said to be {\em perfectly divisible} if for all induced subgraphs $H$ of $G$, $V(H)$ can be partitioned into two sets $A,B$ such that $H[A]$ is perfect and $\omega(B) < \omega(H)$. Ho\`{a}ng \cite{CTH} introduced the notion of perfect divisibility and proved (\cite{CTH}) that (banner, odd hole)-free graphs are perfectly divisible. A nice feature of proving that a graph is perfectly divisible is that we get a quadratic upper bound for the chromatic number in terms of the clique number.  More precisely:

\begin{lemma}\label{QUADRATICCHIBOUND}
Let $G$ be a perfectly divisible graph. Then $\chi(G) \leq {\omega(G) + 1 \choose 2}$.
\end{lemma}

\begin{proof}
Induction on $\omega(G)$. Let $\omega(G) = \omega$. Let $X \subseteq V(G)$ such that $G[X]$ is perfect and $\chi(G \setminus X) < \omega$. Since $G \setminus X$ is perfectly divisible, $\chi(G \setminus X) \leq  {\omega  \choose 2}$. Since $G[X]$ is perfect, $\chi(X) \leq \omega$. Consequently, $\chi(G) \leq \chi(G \setminus X) + \chi(X) \leq \omega +  {\omega  \choose 2} =  {\omega + 1 \choose 2}$. 
\end{proof}

Analogously, $2$-divisibility gives an exponential $\chi$-bounding function. 

\begin{lemma}\label{POWEROFTWOCHIBOUND}
Let $G$ be a $2$-divisible graph. Then $\chi(G) \leq 2^{\omega(G) - 1}$.
\end{lemma}

\begin{proof}
Induction on $\omega(G)$. Let $\omega(G) = \omega$. Let $(A,B)$ be a partition of $V(G)$ such that $\omega(A) < \omega$ and $\omega(B) < \omega$. Now $\chi(A) \leq 2^{\omega - 2}$ and $\chi(B) \leq 2^{\omega - 2}$. Consequently, $\chi(G) \leq \chi(A) + \chi(B) \leq  2^{\omega - 2} +  2^{\omega - 2} =  2^{\omega - 1}$. 
\end{proof}

We end the introduction by setting up the notation that we will be using. For a vertex $v$ of a graph $G$, $N(v)$ will denote the set of neighbors of $v$
(we write $N_G(v)$ if there is a risk of confusion). The closed neighborhood of $v$, denoted $N[v]$, is defined to be $N(v) + v$. We define $M(v)$ (or $M_G(v)$)to be $V(G) \setminus N[v]$.  Let $X$ and $Y$ be disjoint subsets of $V(G)$. We say $X$ is complete to $Y$ if every vertex in $X$ is adjacent to every vertex in $Y$. We say $X$ is anticomplete to $Y$ if every vertex in $X$ is non-adjacent to every vertex in $Y$.  A set $X \subseteq V(G)$ is a {\em homogeneous set} if $1<|X|<|V(G)|$ and every vertex of $V(G) \setminus X$ is either complete or anticomplete to $X$.
If $G$ contains a homogeneous set, we say that $G$ {\em admits a homogeneous set decomposition}. \\

This paper is organized as follows. In section~2 we prove that if a graph contains neither a $P_5$ nor a $C_5$, then it is $2$-divisible.  In Section~3 we prove that if a graph is bull-free and either odd-hole-free or $P_5$-free, then it is perfectly divisible.

\section{$(P_5, C_5)$-free graphs are $2$-divisible}

We start with some definitions. Let $G$ be a graph.  $X \subseteq V(G)$ is said to be {\em connected} if $G[X]$ is connected, and {\em anticonnected} if $G^c[X]$ is connected. For  $X \subseteq V(G)$, a {\em component}  of $X$ is a maximal connected subset of $X$, and an {\em anticomponent} of $X$ is a maximal anticonnected   subset of $X$.

The following lemma is used several times in the sequel. 

\begin{lemma}\label{SEAGULLLEMMA}
Let $G$ be a graph. Let $C \subseteq V(G)$ be connected, and let $v \in V(G) \setminus C$ 
such that $v$ is neither complete nor anticomplete to $C$. Then there exist 
$a,b \in C$ such that $v - a - b$ is a path.  
\end{lemma}

\begin{proof}
Since $v$ is neither complete nor anticomplete to $C$, it follows that both the
sets $N(v) \cap C$ and $M(v) \cap C$ are non-empty. Since $C$ is connected, 
there exist $a \in N(v) \cap C$ and $b \in M(v) \cap C$ such that 
$ab \in E(G)$.  But now $v-a-b$ is the desired path. This completes the proof.
\end{proof}

We are ready to prove the main result of this section.

\begin{theorem} Every $(P_5, C_5)$-free graph is $2$-divisible.  \label{P5C5THEOREM}
\end{theorem}

\begin{proof}
Let $G$ be a $(P_5, C_5)$-free graph. We may assume that $G$ is connected.  Let $v \in V(G)$, let $N=N(v), M=M(v)$. Let $C_1, \cdots , C_t$ be the components of $M$. \\

(1) We may assume that there is $i$ such that no vertex of $N$ is complete to $C_i$. \\

For, otherwise, $X_1=M + v, X_2 = N$ is the desired partition. This proves (1). \\

Let $i$ be as in (1), we may assume  that $i=1$. \\

(2) There do not exist $n_1, n_2$ in $N$ and $m_1, m_2$ in $M$ such that $n_1$ is adjacent to $m_1$ and not
to $m_2$, and $n_2$ is adjacent to $m_2$ and not to $m_1$, and $n_1$ is non-adjacent to $n_2$. \\

For, otherwise, $G[\{n_1, n_2, m_1, m_2, v\}]$ is a $P_5$ or a $C_5$. This proves (2). \\

(3)  For every $i>1$ there exists $n \in N$ complete to $C_i$. \\

For suppose that there does not exist $n \in N$ that is complete to $C_2$. For $i = 1, 2$  let $n_i \in N$ have a neighbor in $C_i$.
Since $C_1, C_2$ are connected, by Lemma \ref{SEAGULLLEMMA}, there exist $a_i, b_i \in C_i$ such that $n_i - a_i - b_i$ is a path.
Since $b_1 - a_1 - n_1 - a_2 - b_2$ is not a $P_5$, we deduce that $n_1 \neq n_2$, and therefore
$n_1$ is complete or anticomplete to $C_2$, and $n_2$ is complete or anticomplete to $C_1$. By
the choice of $C_1$ and the assumption, $n_1$ is anticomplete to $C_2$, and $n_2$ to $C_1$. By (2)
$n_1$ is adjacent to $n_2$. But now $b_2 - a_2 - n_2 - n_1 - a_1$ is a $P_5$, a contradiction. This proves
(3). \\

% Let $n \in N$ have a neighbor in $C_1$, and subject to that choose $n$ with $N(n) \cap M$ maximal.

From the set of vertices in $N$ that have a neighbor in $C_1$, choose one that has the maximum number of neighbors in $M$; call it $n$. (Such a vertex exists because $G$ is connected.)
Let $X_1 = N(n)$, and let $X_2 = V(G) \setminus X_1$. Clearly $X_1$ does not contain a clique of size
$w(G)$.  We claim that $\omega(X_2) < \omega(G)$, thus proving that $(X_1, X_2)$ is a partition
certifying 2-divisibility. \\

Suppose that there is a clique $K$ of size $\omega(G)$ in $X_2$. Then $n \not \in X$. By (3), 
$K \setminus (C_2 \cup \cdots \cup C_t) \neq \emptyset$.  \\

(4) $K \not \subseteq C_1$. \\

For suppose that $K \subseteq C_1$. Then $K \subseteq C_1 \setminus N(n)$. Let $D$ be the component of
$C_1 \setminus N(n)$ containing $K$. Then some vertex $p \in N(n) \cap C_1$ has a neighbor in $D$. Since $D$
contains a clique of size $\omega(G)$, $p$ is not complete to $D$. Since $D$ is connected, by Lemma \ref{SEAGULLLEMMA},
there exist $d_1, d_2 \in D$ such that $p - d_1 - d_2$ is a path. But now $d_2 - d_1 - p - n - v$ is a $P_5$, a contradiction. This proves (4). \\

It follows from (4) that $K$ has a vertex $k_1 \in N \setminus X_1$, and a vertex $k_2 \in M \setminus X_1$.  Then
$k_1$ is non-adjacent to $n$, and $k_2$ is non-adjacent to $n$. But now by (2) $N(k_1) \cap M$
strictly contains $N(n) \cap M$, and in particular $k_1$ has a neighbor in $C_1$, contrary to the
choice of $n$. This completes the proof.
\end{proof}

An easy consequence of this is

\begin{corollary}
Let $G$ be a $(P_5, C_5)$-free graph. Then $\chi(G) \leq 2^{\omega(G) - 1}$.
\end{corollary}

\begin{proof} 
Follows from Theorem \ref{P5C5THEOREM} and Lemma \ref{POWEROFTWOCHIBOUND}
\end{proof}

\section{Perfect divisibility in bull-free graphs}

For an induced subgraph $H$ of a graph $G$, a vertex $c \in V(G) \setminus V(H)$ that is complete to $V(H)$ is called a {\em center} for $H$. Similarly, a vertex $a \in V(G) \setminus V(H)$ that is anticomplete to $V(H)$ is called an {\em anticenter} for $H$.  For a hole $C=c_1 - c_2 - c_3 - c_4 - c_5 - c_1$, an {\em $i$-clone} is a vertex adjacent to $c_{i+1}$ and $c_{i-1}$, and not to $c_{i+2}, c_{i-2}$  (in particular $c_i$ is an $i$-clone).  An {\em $i$-star} is a vertex complete to $V(C) \setminus {c_i}$, and non-adjacent to $c_i$. A {\em clone} is  a vertex that is an $i$-clone for some $i$, and a {\em star} is a vertex that is an $i$-star for some $i$. We will need the following results from \cite{CSa} and \cite{CSi}. \\

\begin{theorem} \label{THM1}
(from \cite{CSi}) If $G$ is bull-free, and $G$ has a $P_4$ with a center and an anticenter, then $G$ admits a homogeneous set decomposition, or $G$ contains $C_5$.
\end{theorem}

\begin{theorem} \label{THM2}
 (from \cite{CSa}) If $G$ is bull-free and contains an odd hole or an odd antihole with a center and an anticenter, then $G$ admits a homogeneous set decomposition.
\end{theorem}

\begin{theorem} \label{THM3}
(from \cite{CSa}) If $G$ is bull-free, then either $G$ admits a homogeneous set decomposition, or for every $v \in V(G)$, either $G[N(v)]$ or
$G[M(v)]$ is perfect.
\end{theorem}

The next two theorems refine Theorem~\ref{THM3} in the special cases we are dealing with in this paper.
\begin{theorem} \label{THM4}
If $G$ is bull-free and odd-hole-free, then either $G$ admits a homogeneous set decomposition, or for every $v \in V(G)$ the graph $G[M(v)]$ is perfect.
\end{theorem}

\begin{proof}
We may assume that $G$ does not admit a homogeneous set  decomposition. Let $v \in V(G)$ such that $G[M(v)]$ is not perfect. Since $G$ is odd-hole-free, by the strong perfect graph theorem \cite{CRST}, $G[M(v)]$ contains an odd antihole of length at least 7, and therefore a three-edge-path $P$ with a center.  Now $v$ is an anticenter for $P$, and so by Theorem \ref{THM1}, $G$ admits a homogeneous set decomposition, a contradiction. This proves the theorem.
\end{proof}

\begin{theorem} \label{THM5}
If $G$ is bull-free and $P_5$-free, then either $G$ admits a homogeneous set decomposition, or for some $v \in V(G)$, $G[M(v)]$ is perfect.
\end{theorem}

\begin{proof}
By Theorem \ref{THM4} we may assume that $G$ contains a $C_5$, say $C=c_1 - c_2 - c_3 - c_4 - c_5 - c_1$. We may assume that $G$ does not admit a homogeneous set decomposition. \\

(1) Let $D$ be a hole of length $5$, and let $v \notin V(D)$. Then $v$ is a clone, a star, a center or an anticenter for $D$. \\

Since $G$ has no $P_5$, $v$ cannot have exactly one neighbor in $D$.  Suppose that $v$ has exactly two neighbors in $D$. Since $G$ is bull-free, the neighbors are non-adjacent, so $v$ is a clone. Suppose that $v$ has exactly two non-neighbors in $D$. Since $G$ is bull-free, the non-neighbors are adjacent, and $v$ is a clone. The cases when $v$ has $0$, $4$, $5$ neighbors in $D$ result in $v$ being an anticenter, star, and a center for $D$, respectively.  This proves (1). \\

(2) Let $D$ be a hole of length $5$ in $G$. Then there is no anticenter for $D$. \\

Suppose that $v$ is an anticenter for $D$, we may assume that $D=C$. By Theorem \ref{THM3} there is no center for $D$. Since $G$ is connected, we may assume that $v$ has a neighbor $u$ such that $u$ has a neighbor in $V(D)$. Let $P$ be a path starting at $u$ and with $V(P) \setminus {u} \subseteq V(D)$ with $|V(P)|$ maximum. Since $v-u-P$ is not a $P_5$, and $v$ is not a center for $P$, it follows that for some $i$, $v$ is adjacent to $c_i$ and to $c_{i+1}$, but not to $c_{i+2}$. But now $G[\{c_i, c_{i+1}, c_{i+2}, u,v\}]$ is a bull,  a contradiction. This proves (2). \\

(3) Let $d_i$ and $d_i'$  be $i$-clones non-adjacent to each other. Let $v$ be adjacent to $d_i$ and not to $d_i'$. Then $v$ is a center for $C$, or $v$ is an $i$-star for $C$,  or $v$ is an $i$-clone for $C$.  Moreover, let  $D$ be the hole obtained from $C$ by replacing $c_i$ with $d_i$, and let 
$D'$ be the hole obtained from $C$ by replacing $c_i$ with $d_i'$.  
It follows that  either 
\begin{itemize}
\item $v$ is an $i$-clone for both $D$ and $D'$, or 
\item $v$ is a center for $D$, and an $i$-star for $D'$.  
\end{itemize}

We may assume that $i=1$. 
If $v$ is anticomplete to $\{c_2,c_5\}$, then we get a contradiction to  (1) or (2) applied to $v$ and $D'$.  Thus we may assume that $v$ is adjacent to $c_2$. Suppose that $v$ is non-adjacent to $c_5$. By (1) applied to $D$, $v$ is adjacent to $c_3$. But now $d_1'-c_5-d_1-v-c_3$ is a $P_5$, a contradiction.  Thus  $v$ is adjacent to $c_5$. By (1) applied to $D'$, $v$ is either complete or anticomplete to $\{c_3, c_4\}$. Now if $v$ is anticomplete to $\{c_3,c_4\}$, then $v$ is an $i$-clone; if $v$ is complete to $\{c_3, c_4\}$ then $v$ is a center or an $i$-star for $C$. This proves (3). \\

(4) There do not exist $d_1, d_1', d_3, d_3', v_1, v_3$ such that 

\begin{itemize}
\item $\{d_1, d_1'\}$ is not complete to $\{d_3, d_3'\}$, and 
\item for $i={1,3}$
\begin{itemize}
\item $d_i$ and $d_i'$  are $i$-clones non-adjacent to each other, and
\item $v_i$ is adjacent to $d_i$ and non-adjacent to to $d_i'$, and 
\item $v_i$ is not an $i$-clone.  
\end{itemize}
\end{itemize}

Observe that by (3), no vertex of $\{d_1,d_1'\}$ is mixed on $\{d_3,d_3'\}$ and the same with the roles of $1,3$ exchanged. It follows that $\{d_1,d_1'\}$ is anticomplete to $\{d_3,d_3'\}$, and in particular $v_1, v_3 \not \in \{d_1,d_1',d_3,d_3'\}$. By (3) applied to the hole $d_1'-c_2-c_3-c_4-c_5-d_1'$ and ${d_3,d_3'}$, it follows that $v_3$ is complete to $\{d_1,d_1'\}$. Similarly $v_1$ is complete to $\{d_3,d_3'\}$. In particular $v_1 \neq v_3$. But now $G[\{d_1', v_3, d_1, v_1, d_3'\}]$ is either a bull or a $P_5$, in both cases a contradiction. This proves (4). \\

(5) There is not both a $1$-clone non-adjacent to $c_1$, and a $3$-clone non-adjacent to $c_3$. \\

For suppose that such clones exist. For $i=1,3$ let $X_i$ be a maximal anticonnected set of $i$-clones with $c_i$ in $X_i$. Then $|X_i| > 1$ for $i=1,3$. 
Since $X_i$ is anticonnected, it follows from
(3) that $X_1$ is anticomplete to $X_3$. Since $|X_1|, |X_3| > 1$, and $G$ does not admit a homogeneous set decomposition, it follows 
that neither $X_1$ nor  $X_3$ is a homogeneous set in $G$. Therefore for $i = 1,3$ there 
exists $v_i \not \in X_i$ with a neighbor and a non-neighbor in $X_i$. Then $v_i  
\not \in X_1 \cup X_3$. Note that $X_i + v_i$ is anticonnected, and hence by the maximality of $X_i$, it follows that $v_i$ is not an 
$i$-clone.  By applying Lemma \ref{SEAGULLLEMMA} in $G^c$ with $v_i$ and $X_i$ for $i =1,3$,  it follows that there exist $d_i, d_i' \in 
X_i$ such that  $d_i$ is non-adjacent to $d_i'$, $v_i$ is adjacent to $d_i$, and  $v_i$ is non-adjacent to  
$d_i'$. But now we get a contradiction to (4). This proves (5).  \\

(6) For some $i$, $V(G)=N[c_i] \cup N[c_{i+2}]$ (here addition is {\em mod 5}). \\

Suppose that (6) is false. 
Since (6) does not hold with $i=1$, (1), (2)  and symmetry  imply that we may assume that there is a $1$-clone $c_1'$ non-adjacent to $c_1$. 
Since  (6) does not hold with $i=5$, again by (1), (2) and symmetry we may assume that there is a $2$-clone $c_2'$ non-adjacent to $c_2$. 
Finally, since  (6) does not hold with $i=3$, by (1), (2) and symmetry we get  a $3$-clone $c_3'$ non-adjacent to $c_3$. But this is a contradiction to (5). This proves (6). \\

Let $i$ be as in (6); we may assume that $i=1$. Suppose that $G[M(c_1)]$ is not perfect. Then, by the strong perfect graph theorem \cite{CRST}, $G[M(c_1)]$ contains an odd hole or an odd antihole $H$. But now $c_3$ is a center for $H$, and $c_1$ is an anticenter for $H$, contrary to Theorem \ref{THM2}. This proves the theorem.
\end{proof}

A graph $G$ is {\em perfectly weight divisible} if for every non-negative integer weight function $w$ on $V(G)$,  there is a partition of $V(G)$ into two sets $P,W$ such that $G[P]$ is perfect and the maximum weight of a clique in $G[W]$ is smaller than the maximum weight of a clique in $G$. \\

\begin{theorem} \label{THM6}
A minimal non-perfectly weight divisible graph does not admit a homogeneous set decomposition.
\end{theorem}

\begin{proof} Let $G$ be such that all proper induced subgraphs of $G$ are perfectly weight divisible. Let $w$ be a weight function on $V(G)$. Let $X$ be a homogeneous set in $G$, with common neighbors $N$ and let $M = V(G) \setminus (X \cup N)$. Let $G'$ be obtained from $G$ by replacing $X$ with a single vertex $x$ of $X$ with weight $w(x)$ equal to the maximum weight of a clique in $G[X]$. Let $T $ be the maximum weight of a clique in $G$. \\

Let $(P',W')$ be a partition of $V(G')$ corresponding to the weight $w$. Let $(X_p, X_w)$ be a partition of $X$ where $G[X_p]$ is perfect and the maximum weight of a clique in $G[X_w]$ is smaller than the maximum weight of a clique in $G[X]$. We construct a partition of $V(G)$. \\

Suppose first that $x \in W'$. Then let $P=P'$ and $W=W' \cup X$. Clearly this is a good partition. Now suppose that $x \in P'$. Let $P = (P'  \setminus x)  \cup X_p$ and let $W=W' \cup X_w$. By a theorem of \cite{LL}, $G[P]$ is perfect.
Suppose that $W$ contains a clique $K$ with weight  $T$. Then $K \cap X_w$ is non-empty. Let $K'$ be a clique of maximum weight in $X$. Now $(K \setminus X_w)  \cup K'$ is a clique in $G$ with weight  greater than $T$,  a contradiction. This proves the theorem.
\end{proof}

We can now prove our main result:

\begin{theorem}\label{MAINTHEOREM}
Let $G$ be a bull-free graph that is either odd-hole-free or $P_5$-free.
Then $G$ is perfectly weight divisible, and hence perfectly divisible.
% and let $w$ be a weight function on $G$.
\end{theorem}

\begin{proof}
Let $G$ be a minimal counterexample to the theorem. Then there is a
non-negative integer weight function $w$ on $V(G)$ for which 
there is no partition of $V(G)$ as in the definition of being perfectly 
weight divisible. Let $U$ be the set of vertices of $G$ with $w(v)>0$,
and let $G'=G[U]$.  
By theorems \ref{THM4}, \ref{THM5}, \ref{THM6}, $G'$ has a vertex $v$ such that $G'[M_{G'}(v)]$ is perfect. But now, since $w(v) >0$,  setting 
$P=M_{G'}(v)+v$ and 
$W=N_{G'}(v) \cup (V(G) \setminus U)$ we get a partition of $V(G)$ as
in the definition of being perfectly weight divisible, a contradiction.
This proves the theorem.
\end{proof}

\begin{corollary}
Let $G$ be a bull-free graph that is either odd-hole-free or $P_5$-free. Then $ \chi(G) \leq {\omega(G) + 1 \choose 2}$. 
\end{corollary}

\begin{proof}
Follows from Theorem \ref{MAINTHEOREM} and Lemma \ref{QUADRATICCHIBOUND}.
\end{proof}

\section{Acknowledgment}
This research was performed during the 2017 Barbados Graph Theory Workshop at the McGill University
Bellairs Research Institute in Barbados, and the authors are grateful to the institute for
its facilities and hospitality. The authors also thank Ch\'{i}nh T. Ho\`{a}ng for telling them
about these problems, and for many useful discussions.


\begin{thebibliography}{99}
%\bibitem{MC} M. Chudnovsky,  The structure of bull-free graphs I --- Three-edge-paths with centers and  anticenters,   {\em Journal of Combinatorial Theory. Ser B},  {\bf 102} (2012), 233-251.
% \bibitem{MC} M. Chudnovsky, The structure of bull-free graphs I - three-edge paths with centers and anticenters, Journal of Combinatorial Theory, Ser. B, 102 (2012), 233-251.

\bibitem{CRST} M. Chudnovsky, N. Robertson, P. Seymour, R. Thomas, The strong perfect graph theorem, Annals of Mathematics, 164 (2006), 51-229	. 
\bibitem{CSa} M. Chudnovsky, S. Safra, The Erd\H{o}s-Hajnal Conjecture for bull-free graphs, Journal of Combinatorial Theory, Ser. B, 98 (2008), 1301-1310.
\bibitem{CSi} M. Chudnovsky, V. Sivaraman, Odd holes in bull-free graphs, submitted.
\bibitem{CTH} C. T. Ho\`{a}ng, On the structure of (banner, odd hole)-free graphs, submitted. (Available at https://arxiv.org/abs/1510.02324)
\bibitem{HM} C. T. Ho\`{a}ng, C. McDiarmid, On the divisibility of graphs, Discrete Math. 242 (2002), no. 1-3, 145-156.
\bibitem{LL}  L. Lov\'{a}sz,  Normal hypergraphs and the perfect graph conjecture, Discrete Math. 2 (1972), no. 3, 253-267. 

\end{thebibliography}
\end{document}